\documentclass{amsart}%
\usepackage{amssymb}
\usepackage{amsfonts}%
\usepackage{amsmath}%
\usepackage{graphicx}
\newtheorem{theorem}{Theorem}
\theoremstyle{plain}

\newtheorem{corollary}{Corollary}

\newtheorem{definition}{Definition}

\newtheorem{lemma}{Lemma}

\numberwithin{equation}{section}
\begin{document}
\title[Levelt Theorem]{On the Levelt's Theorem}

\begin{abstract}
Let $(E)$ be a homogeneous linear differential equation Fuchsian of order $n$
over $\mathbb{P}^{1}\left(  \mathbb{C}\right)  $. The idea of Riemann (1857)
was to obtain the properties of solutions of ($E$) by studying the local
system. Thus, he obtained some properties of Gauss hypergeometric functions by
studying the associated rank $2$ local system over $\mathbb{P}^{1}\left(
\mathbb{C}\right)  \backslash\left\{  3\ points\right\}  $. For example, he
obtained the Kummer transformations of the hypergeometric functions without
any calculation. The success of the Riemann's methods is due to the fact that
the irreducible rank $2$ local system over $\mathbb{P}^{1}\left(
\mathbb{C}\right)  \backslash\left\{  3\ points\right\}  $ is linearly "rigid"
in the sense of Katz \cite{Katz}. This result constitute one of the best
studied example of linear rigid system, it was proved by the Levelt's theorem
\cite{B} Theorem 1.2.3. In this work we propose a partial generalization of
the Levelt's theorem.

\end{abstract}
\author{Lotfi Saidane}
\address{Lotfi Saidane, D\'{e}partement de Math\'{e}matiques, Facult\'{e} des sciences
de Tunis, Campus Universitaire, 1060 Tunis, Tunisie.}
\email{lotfi.saidane@fst.rnu.tn}
\date{October 2009}
\maketitle

\section{Introduction}

\noindent Let $n\in\mathbb{N}_{\geq1}$ and $a_{1},$ .., $a_{n}\in
\mathbb{C(}z).$ We consider the homogeneous linear differential equation of
order $n$ on $\mathbb{P}^{1}\mathbb{(C})$ :%
\begin{equation}
y^{(n)}+a_{1}y^{(n-1)}+....+a_{n-1}y^{\prime}+a_{n}y=0.\tag{E}%
\end{equation}
We denote by $S=\left\{  \varpi_{1},..,\varpi_{s}\right\}  $ the non empty
set, in $\mathbb{P}^{1}\mathbb{(C})$, of its singularities, we required that
will be all regular. We fix a base point $z_{0}\in\mathbb{P}^{1}%
\mathbb{(C})\backslash S$ and denote by $G$ the fundamental group $\pi
_{1}(\mathbb{P}^{1}\mathbb{(C})\backslash S,\ z_{0})$. Then $G$ is a free
group generated by the classes of homotopy of loops $\gamma_{i}$ starting at
$z_{0}$ and making a turn in the direct sense of $\varpi_{i},$ in a
neighborhood not containing $\varpi_{j},$ $j\neq i,$ then returning to $z_{0}$
such that :
\[%
{\textstyle\prod\nolimits_{i\in\left\{  1,..,s\right\}  }}
\gamma_{i}=1.
\]
Because $z_{0}$ is a regular point for the equation $(E)$, the Cauchy
conditions are satisfied, therefore there are $n$ solutions (local) of $(E)$
holomorphic near $z_{0}$, linearly independent over $\mathbb{C}$. Let $V$ be a
$\mathbb{C-}$vector space spanned by these solutions. The group representation%
\[
M_{(E)}:\pi_{1}(\mathbb{P}^{1}\mathbb{(C})\backslash S,\ z_{0})\rightarrow
GL(V),
\]
is called the monodromy representation of $(E).$ For $i\in\left\{
1,..,s\right\}  $, we put $M_{i}=M_{(E)}(\gamma_{i}).$ Then, $M_{i}\in
GL_{n}(\mathbb{C})$ and
\[%
{\textstyle\prod\nolimits_{i\in\left\{  1,..,s\right\}  }}
M_{i}=I_{n}.
\]
The matrices $M_{i}$ are called (local) monodromy matrix, they constitute a
local (complex) system of order $n$ on $\mathbb{P}^{1}\mathbb{(C})\backslash
S$. The group generated by the matrices $M_{i}$ is called the "monodromy
group" of $(E)$, related to the basis of local solutions at $z_{0}$.

\noindent\textbf{Question}: If we change the basis of local solutions, is the
new monodromy group isomorphic to the former?

\noindent In other words: The local system defined by the $M_{i}$ is it
"linearly rigid" (in the sens below)?

\noindent Let $r\in\mathbb{N}_{\geq2}$ and $g_{1}$, $g_{2}$, $....$, $g_{r}$
elements of $GL_{n}(\mathbb{C})$ satisfying%
\[
g_{1}.g_{2}....g_{r}=Id_{n}%
\]
We say that the $r-$tuple $\left\{  g_{1}\text{, }g_{2}\text{, }...\text{,
}g_{r}\right\}  $ is linearly rigid if for any conjugate $\widetilde{g_{1}}$,
$\widetilde{g_{2}}$, $...$, $\widetilde{g_{r}}$ of $g_{1}$, $g_{2}$, $....$,
$g_{r}$ in $GL_{n}(\mathbb{C})$ satisfying:%
\[
\breve{g}_{1}\breve{g}_{2}...\breve{g}_{r}=Id_{n},
\]
there is $u$ in $GL_{n}(\mathbb{C}$ $)$ such that $\breve{g}_{i}=ug_{i}u^{-1}
$ for $i=1,2,...r$. For example, the couple $(g$, $g^{-1})$, $g\in
GL_{n}(\mathbb{C})$ is linearly rigid.

\noindent The group $<g_{1},..,g_{r}>$ is said irreducible if and only if it
acts irreducibly on $\mathbb{C}^{n}.$ Katz \cite{Katz} theorem 1.1.2,
characterized the Jordan normal forms of irreducible linearly rigid local
systems. The Levelt theorem, \cite{B} Theorem 1.2.3 shows that the local
system associated to a hypergeometric equation is irreducible linearly rigid.

\section{Pseudo-reflection}

\begin{definition}
We say that $h\in GL(n,\mathbb{C)}$ is a pseudo-reflection if the rank of
$(h-Id_{n})$ is $1.$
\end{definition}

\begin{lemma}
Let $n,p\in\mathbb{N}_{\geq2}$ and $A_{1},$ $A_{2},$ ..., $A_{p}\in$
$GL_{n}(\mathbb{C)}$ such that, for all $i,$ $j\in\left\{  1,2,...,p\right\}
$, $i<j$, the operator $A_{i}A_{j}^{-1}$ is a pseudo-reflection. Then, up to
conjugation, $A_{1},$ $A_{2},$ ..., $A_{p}$ have the same $(n-1)$ first rows
or columns.\label{ligne}
\end{lemma}

\begin{proof}
We will furnish a proof for $p=3$, the general case being similar. Assume that
$n\geq3.$\newline We set $W_{1}=\ker(A_{1}-A_{2})$ and $W_{2}=\ker(A_{2}%
-A_{3}).$ Because $A_{1}A_{2}^{-1}$ and $A_{2}A_{3}^{-1}$ are
pseudo-reflections, we deduce that $W_{1}$ and $W_{2}$ are $n-1$ dimensional
subspaces of $\mathbb{C}^{n}$. If $W_{1}=W_{2},$ we choose a basis of $W_{1}$
and we complete it by one vector to obtain a basis of the total space. With
respect to this basis the matrices $A_{1}$, $A_{2}$ and $A_{3}$ have the same
$(n-1)$ first columns.\newline Assume that $W_{1}\neq W_{2}.$ Because
$n\geq3,$ the vector space $W_{1}\cap W_{2}$ has dimension $n-2.$ Let
$\left\{  e_{1},...,e_{n-2}\right\}  $ a basis of the space $W_{1}\cap W_{2}%
$.\newline* If $A_{1}A_{2}^{-1}$ is a reflection, there is $e_{n}$ such that
$(A_{1}-A_{2})(e_{n})=\nu e_{n},$ with $\nu\neq0.$ We choose $e_{n-1}$ in
$W_{1}$ so that $\left\{  e_{1},...,e_{n-1}\right\}  $ is a basis of $W_{1}.$
The system $\underline{e}=\left\{  e_{1},...,e_{n}\right\}  $ is then a basis
of $\mathbb{C}^{n}.$ For $i\in\left\{  1,2,3\right\}  ,$ $j\in\left\{
1,2,..,n\right\}  ,$ We denote by $A_{i,j}$, the $j-th$ column of the matrix
$A_{i}$, relatively to the basis $\underline{e},$ Because $rank(A_{1}%
-A_{3})=rank(A_{2}-A_{3})=1$ and $W_{1}\neq W_{2},$ there exist $\lambda$ and
$\beta$ non-zero in $\mathbb{C}$ such that
\[
A_{1,n-1}-A_{3,n-1}=\lambda(A_{1,n}-A_{3,n})
\]
and%
\[
A_{2,n-1}-A_{3,n-1}=\beta(A_{2,n}-A_{3,n}).
\]
We have, by hypothesis $A_{1,n-1}=A_{2,n-1}$ and $A_{1}e_{n}=A_{2}e_{n}+\nu
e_{n}.$ Substituting these relations in the previous equalities, we obtain, by
abuse of writing,%
\[
\lambda(A_{2,n}+\nu e_{n}-A_{3,n})=\beta(A_{2,n}-A_{3,n}).
\]
Therefore, there is at least $\alpha\in\mathbb{C}$, $\alpha\neq0$ such that
\[
A_{2,n}-A_{3,n}=\alpha e_{n},
\]
then%
\[
A_{1,n-1}-A_{3,n-1}=\lambda\alpha e_{n},
\]
and consequently, the $\left(  n-1\right)  $ first row of the matrices
$A_{1},$ $A_{2}$ and $A_{3}$ are identical.\newline* If $A_{1}A_{2}^{-1}$ is
idempotent ($1$ is the unique eigenvalue), then the image of $A_{1}-A_{2}$ is
contained in its kernel $W_{1}$. Let $w$ be a generator of $\operatorname{Im}%
\left(  A_{1}-A_{2}\right)  ,$ then, there exist a vector $e_{n}$ in
$\mathbb{C}^{n}$ such that $\left(  A_{1}-A_{2}\right)  e_{n}=w.$ if $w\in
W_{1}\cap W_{2},$ we set $e_{1}=w$ so that $\left\{  e_{1},...,e_{n-2}%
\right\}  $ is a basis of the space $W_{1}\cap W_{2}.$ We choose $e_{n-1}$ in
$W_{1}$ such that $\left\{  e_{1},...,e_{n-1}\right\}  $ is a basis of
$W_{1}.$ Then, the system $\underline{e}=\left\{  e_{1},...,e_{n}\right\}  $
is a basis of $\mathbb{C}^{n}.$ If $w\notin W_{1}\cap W_{2},$ we ste
$e_{n-1}=w.$ Thus, there exist a unique index $m\in\left\{  1,n-1\right\}  ,$
such that $w=e_{m}.$ Under these conditions, the $(n-1)$ first columns of
$A_{1}-A_{2}$ are zero and the last column is equal to $e_{m}.$ The $(n-2$
first columns of $A_{1}-A_{3}$ and $A_{2}-A_{3}$ are zero. Using the fact that
$A_{1}-A_{3}$ and $A_{2}-A_{3}$ are of rank $1$, we obtain that all components
of their two last columns, except the $m-th$ row, are zero. Therefore, exept
the $m-$th row, all rows of $A_{1}$, $A_{2}$ and $A_{3}$ are
identical.\newline If $n=2$ and $W_{1}\neq W_{2}.$ We repeated the same
argument by replacing $W_{1}\cap W_{2}$ by $\left\{  0\right\}  .$
\end{proof}

\noindent The two following results modify and generalize, in part, the
theorem 1.2.1 of \cite{B}.

\begin{theorem}
Let $n,p\in\mathbb{N}_{\geq2}$ and $A_{1},$ $A_{2},$ ..., $A_{p}\in$
$M_{n}(\mathbb{C)}$ having the same $(n-1)$ first rows or columns and one
eigenvalue in commun. Then these matrices stabilize at least a line or a
hyperplane of $\mathbb{C}^{n}.\label{stabilise}$
\end{theorem}

\begin{proof}
We note that a system of matrices stabilizing the same hyperplane if and only
if their transpose matrices, as endomorphism of the dual space, stabilize the
same line, which mean they have an eigenvector in common. Without loss of
generality, we may assume that $(n-1)$ first rows of the matrices $A_{i}$ are
identical. If $\lambda$ denotes the common eigenvalue, then the matrices
$A_{1}-\lambda I_{n}$, $A_{2}-\lambda I_{n},$ ..., $A_{p}-\lambda I_{n}$ has
the same $(n-1)$ first rows and have rank less than equal $n-1$ :\newline- If
the $(n-1)$ first rows of these matrices are linearly independent, then the
last row of each of these matrices is a linear combination of the previous.
Let $v$ be a nonzero vector orthogonal to the $(n-1)$ first rows. Then $v$ is
orthogonal to the last row of each of these matrices. Thus, $v$ is a common
eigenvector with eigenvalue $\lambda,$ of the matrix $A_{1},$ .., $A_{p}%
.$\newline- If the $(n-1)$ first rows of these matrices are linearly
dependent, then the $(n-1)$ first columns of the transpose matrices are
linearly dependent. Let $c_{1},$ $c_{2},$ ..., $c_{n-1}$ be the coefficients
of a non trivial linear dependence relation. Thereby, $v=(c_{1},$ $c_{2},$
..., $c_{n-1},$ $0) $ is a common eigenvector of all matrices $A_{1}^{T},$ ..,
$A_{p}^{T}$ with eigenvalue $\lambda$. Therefore, the matrices $A_{1},$ ..,
$A_{p}$ stabilize, simultaneously, a hyperplane.\newline If the matrix
$A_{1},$ .., $A_{p}$ have $(n-1)$ common columns , then their transpose have
$(n-1)$ ommon rows, thus the above reasoning leads to the conclusion.
\end{proof}

\begin{theorem}
Let $n,p\in\mathbb{N}_{\geq2}$, $A_{1},$ $A_{2},$ ..., $A_{p}\in$
$M_{n}(\mathbb{C)}$ having the same $(n-1)$ first rows or columns and
stabilize a same non trivial subspace of $\mathbb{C}^{n}.$ Then, $\cap
_{i=1}^{p}specA_{i}\neq\emptyset.\label{spec}$
\end{theorem}

\begin{proof}
We may suppose that, relatively to some basis $B=\left\{  e_{1},..,e_{n}%
\right\}  $ of $\mathbb{C}^{n}$, the matrix $A_{1},$ ..., $A_{p} $ have the
same $(n-1)$ first columns. Denote by $E$ the subspace of $\mathbb{C}^{n}$
generated by $\left\{  e_{1},..,e_{n-1}\right\}  $. Let $W$ be a nontrivial
subspace of $\mathbb{C}^{n}$ stable under the action of $A_{i}$. Suppose that
$W\subset E$ and $\dim_{\mathbb{C}}W=r\in\left\{  1,..,n-1\right\}  $. Let
$\left\{  w_{1},..,w_{r}\right\}  $ a basis of $W$, we complete so that
$B^{\prime}=\left\{  w_{1},..,w_{n}\right\}  $ is a basis of $\mathbb{C}^{n}$.
Since $A_{i}e_{j}=A_{k}e_{j},$, for $i,k\in\left\{  1,..,p\right\}  $ and
$j\in\left\{  1,..,n-1\right\}  $, we deduce that
\[
A_{i}w_{j}=A_{k}w_{j},~\text{for }~i,k\in\left\{  1,..,p\right\}
\ \text{and}\ j\in\left\{  1,..,r\right\}  ,
\]
This proves that, knowing the stability of $W$, and relatively to the basis
$B^{\prime}$ the matrix $A_{i}$ have the form below
\[
A_{i}=\left(
\begin{array}
[c]{cc}%
A_{(r,r)} & \ast_{(r,n-r)}\ \ \ \ \\
o & \ast_{(n-r,n-r)}%
\end{array}
\right)  ,
\]
where $A_{(r,r)}$ is an order $r$ matrix common to all $A_{i},$ $\ast
_{(r,n-r)}$ (resp. $\ast_{(n-r,n-r)})$ is an element of $M_{(r,n-r)}%
(\mathbb{C})$ (resp. $M_{(n-r,n-r)}(\mathbb{C})).$. Therefore, the polynomial
$\det(A_{(r,r)}-\lambda I_{r})$ divides all the characteristic polynomials of
all $A_{i}$. The complex roots of $\det(A_{(r,r)}-\lambda I_{r})$ are in
$\cap_{i=1}^{p}specA_{i}.$\newline Assume that $W\nsubseteq E,$ then there is
a basis $\left\{  f_{n-p+1},..,f_{n}\right\}  $ of $W$ and a free system
$\left\{  g_{1},..,g_{n-p}\right\}  $ of $E$ such that $\left\{
g_{1},..,g_{n-p},f_{n-p+1},..,f_{n}\right\}  $ is a basis of $\mathbb{C}^{n}$.
With respect to this basis the matrix $A_{i}$ have the form:%
\[
A_{i}=\left(
\begin{array}
[c]{c}%
A_{(n,n-p)}%
\end{array}%
\begin{array}
[c]{c}%
0\ \ \ \ \ \ \\
\ \ast_{(p,p)}%
\end{array}
\right)  ,
\]
where $A_{(n,n-p)}$ is a matrix with $n$ rows and $(n-p)$ columns common to
all $A_{i}$ and $\ast_{(p,p)}$ is a matrix of order $p$. Thus, we deduce that
the $A_{i}$ have at least one common eigenvalue.
\end{proof}

\begin{corollary}
[Beukers Th\'{e}or\`{e}me 1.2.1]Let $H$ be a subgroup of $GL_{n}(\mathbb{C)}$
generated by two matrices $A$ and $B$ satisfying $AB^{-1}$ is a
pseudo-reflection. Then $H$ is linearly irreducible if and only if, the
spectrum of $A$ and $B$ are linealy disjoint.
\end{corollary}

\begin{proof}
The lemma \ref{ligne} for $p=2,$ prooves that $A$ and $B$ have $(n-1)$ rows or
columns in commun. The theorems \ref{stabilise} and \ref{spec} furnished the result.
\end{proof}

\section{Levelt's Theorem}

\noindent The following result modifies, slightly, and generalizes in part,
the Levelt's theorem for case $p\geq2$ (see \cite{B} theorem 1.2.3).

\begin{theorem}
Let $n,p\in\mathbb{N}_{\geq2}$ and $\alpha_{i}=\left\{  \alpha_{i,1}%
,...,\alpha_{i,n}\right\}  \subset\mathbb{C}^{\ast},$ $1\leq i\leq p, $
satisfaying $\cap_{i=1}^{p}\alpha_{i}=\emptyset.$ Then, ther exist $A_{1},$
$A_{2},$ ..., $A_{p}$ in $GL_{n}(\mathbb{C)}$ having the same $(n-1)$ first
columns (unique up to a same isomorphism) such that for every $i,$
$specA_{i}=\alpha_{i}.\label{lev}$
\end{theorem}

\begin{proof}
Existence :\newline For every $(i,j)\in\left\{  1,..,p\right\}  \times\left\{
1,..,n\right\}  ,$ we define $A_{i,j}$ by
\[%
{\textstyle\prod\nolimits_{j=1}^{n}}
(X-\alpha_{i,j})=X^{n}+%
{\textstyle\sum\nolimits_{k=0}^{n-1}}
A_{i,n-k}X^{k},\
\]
By hypothesis, the $\alpha_{i,j}$ are non zero, than the matrices defined by
\[
A_{i}=\left(
\begin{array}
[c]{ccc}%
0\ 0 & 0 & -A_{i,n}\\
1\ 0 &  & .\\
0\ 0 & 1 & -A_{i,1}%
\end{array}
\right)  ,
\]
are in $GL(n,\mathbb{C)}$ The characteristic polynomial of $A_{i}$ is
\begin{align*}
\det(XI_{n}-A_{i}) &  =X^{n}+%
{\textstyle\sum\nolimits_{k=0}^{n-1}}
A_{i,n-k}X^{k}\\
&  =%
{\textstyle\prod\nolimits_{j=1}^{n}}
(X-\alpha_{i,j}).
\end{align*}
which proves the existence.\newline Uniqueness :\newline Let $A_{1},$ $A_{2},$
..., $A_{p}\in GL(n,\mathbb{C)}$ having the same $(n-1) $ first columns. Let
$\left\{  e_{1},..,e_{n}\right\}  $ be a basis of $\mathbb{C}^{n}$, relatively
to it we have, for all $i,j\in\left\{  1,..,p\right\}  ,$ $k\in\left\{
1,..,n-1\right\}  $ :%
\[
A_{i}e_{k}=A_{j}e_{k}.
\]
We denote by $W$ the vectoriel suspace of $\mathbb{C}^{n}$ generated by
$\left\{  e_{1},..,e_{n-1}\right\}  .$ We have :%
\[
\dim_{\mathbb{C}}(W\cap A_{1}W\cap...\cap A_{1}^{n-2}W)\geq1.
\]
We suppose that the dimensional of $W\cap A_{1}W\cap...\cap A_{1}^{n-2}W$ is
greater than or equal to $2$. Therefore, we have
\[
\dim_{\mathbb{C}}(W\cap A_{1}W\cap...\cap A_{1}^{n-1}W)\geq1.
\]
Hence, the subspace $W\cap A_{1}W\cap...\cap A_{1}^{n-1}W$ of $\mathbb{C}^{n}$
is nontrivial and stable under the action of all $A_{i}$. Theorem \ref{spec}
shows that $\cap_{i=1}^{p}specA_{i}\neq\emptyset$ which is absurd. So%
\[
\dim_{\mathbb{C}}(W\cap A_{1}W\cap...\cap A_{1}^{n-2}W)=1,
\]
Then, there is a vector $v$ in $W$ such that the system $\left\{
v,A_{1}v,..,A_{1}^{n-2}v\right\}  $ is a basis of $W$, we complete with a
vector to a basis of the total space $\mathbb{C}^{n}$. With respect to the
latter matrices $A_{i}$ have the form
\[
A_{i}=\left(
\begin{array}
[c]{ccc}%
0\ 0 & 0 & -A_{i,n}\\
1\ 0 &  & .\\
0\ 0 & 1 & -A_{i,1}%
\end{array}
\right)  ,
\]
where $A_{i,j}$ are determined by the spectrum $\left\{  \alpha_{i,1}%
,...,\alpha_{i,n}\right\}  $ of $A_{i}$ as follows:%
\[%
{\textstyle\prod\nolimits_{j=1}^{n}}
(X-\alpha_{i,j})=X^{n}+%
{\textstyle\sum\nolimits_{k=0}^{n-1}}
A_{i,n-k}X^{k},
\]
which completed the proof.
\end{proof}

\end{document}